\documentclass[a4paper,12pt]{article}
\usepackage{amsmath}
\usepackage{amsfonts}
\usepackage{amssymb}
\usepackage{amsthm}
\usepackage{graphicx}
\usepackage{ifthen}
\usepackage{verbatim}
\usepackage{amsmath}
\usepackage{setspace}
\usepackage{amsfonts}
\usepackage{amssymb}
\usepackage{amsthm}
\usepackage{lscape}
\usepackage{setspace}
\usepackage{listings}
\usepackage{amsmath}
\usepackage{amsfonts}
\usepackage{amssymb}
\usepackage{amsthm}
\usepackage{ifthen}
\usepackage{verbatim}
\newtheorem{thm}{Theorem}[section]

\newtheorem{coro}[thm]{Corollary}
\newtheorem{lem}[thm]{Lemma}
\newtheorem{defn}[thm]{Definition}
\newtheorem{rem}[thm]{Remark}
\newtheorem{prop}[thm]{Proposition}

\def\1{\emph{\textbf{1}}}
\def\R{\mathbb{R}}
\def\P{\mathbb{P}}

\def\E{\mathbb{E}}
\def\N{\mathbb{N}}

\def\Z{\mathbb{Z}}
\def\T{\mathcal{T}}

\def\F{\mathcal{F}}

\begin{document}

\title{From Triangular Array Progression\\
to Near-Log-Concave State Transitions}
\author{Levent Ali Meng\"ut\"urk\thanks{University College London, Department of Mathematics, ucaheng@ucl.ac.uk, and Artificial Intelligence and Mathematics Research Lab, levent@aimresearchlab.com} \\ \small University College London and AIM Research Lab
}
\date{}
\maketitle
\begin{abstract}
The paper introduces near-log-concave random variables and the state-transitions of a family of Markov chains derived from an algorithmic integer sequence progression. We propose a combinatorial rule that generates an asymmetric triangular array whose $n$-th row sums to $2^{n}$ for all $n \geq 0$, where every row is the degree sequence of a multigraph without loops. The structure hosts infinitely many zero-free unimodal near-log-concave sequences whose log-concavity deviation converges to $\log(4/3)$ under logarithmic scaling. From a probabilistic perspective, the construct defines a sequence of near-log-concave probability mass functions, from which, a structural decoupling of moments emerges: the mean diverges while the variance asymptotically converges to a steady-state limit. We study the Shannon entropy dynamics of the triangular array, and generate infinitely many right-stochastic matrices via coordinate transformations over the Tychonoff cube. Our specific framework indicates a broader methodology for modelling asymmetric stochastic systems, where distributions propagate as non-dispersive discrete wave packets, advancing indefinitely without flattening.
\end{abstract}

\section{Introduction}
The main direction of this paper is to introduce and study an infinite triangular array generated through a particular sequence progression algorithm. We keep this paper compact and computationally tractable to provide a list of results that shed light on how the proposed construct behaves. As a brief background to its emergence, the author uncovered the triangular array sequence while working on a broader probability problem involving systems that satisfy the following set of state-transition characteristics:
\begin{enumerate}
\item From any state, the system has the highest probability to remain in that state in the next time step when compared to any other one state
\item From any state, transition probabilities are monotonically non-increasing from the closest previously-visited state to progressively distant previously-visited states in the next time step
\item From any state, the system can transition \emph{only} to the closest unvisited state in the next time step
\end{enumerate}
As an overview, such a stochastic system has the highest likelihood to stay where it currently is when compared to transitioning to another state. If this system transitions to another previously achieved state, the transitioned state is more likely to neighbor the current state than to being further away from it. Finally, the system visits only the closest new state and no other previously unexplored state is achieved in one time step. Such dynamical systems, either partially or fully, may be observed in atomic phase transitions, social queuing, cellular cycles and gene expressions -- see \cite{0app,3app,4app,5app,5app1,6app,6app1,7app} and references therein. For example, in quantum mechanics and atomic physics, such patterns may arise in systems wherein electrons jumping to lower energy states (i.e. more ordered states) is more probable than leaping to a higher energy state (i.e. a more disordered state); e.g. quantum random walks may exhibit similar left-skewed, asymmetric behaviour (see \cite{3q,7q}). While this work does not claim that the proposed structure is fit to model these specific empirical fields, we shall provide a method to transform the given triangular array, which we use as an \emph{anchor} integer structure, into infinitely many transition matrices that satisfy all the three aforementioned characteristics. To the best of our knowledge, the construct presented in this paper does not exist in the literature.

We show that the introduced anchor triangular array hosts infinitely many unimodal \emph{near}-log-concave sequences with no zeros. In doing so, we first prove that there exists an infinite \emph{sub}-triangular array, -- each row a strictly-positive log-concave sequence (i.e. a Pólya frequency sequence of order 2) -- that almost entirely dominates the full triangular array, in which log-concavity fails only at a \emph{single} early column. At this stage, we shall highlight that log-concave sequences arise in numerous areas of mathematics including algebraic graph theory, combinatorics and geometry -- see \cite{0aa,1a,2,3,4,5,6,7,8,9,10,11,12}, amongst many other important works. Next, we prove that every row of the triangular array sums to an increasing power of 2, where the power is determined in terms of the cardinality of that row; more precisely, the $n$th row of the array sums to $2^n$ for every $n\geq0$. We show that every row of the array defines the degree sequence of a multigraph without loops; see \cite{graphone}. Although the first four rows of the proposed triangular progression coincide with those from Pascal's triangle, its elements become asymmetric thereafter and diverge from binomial coefficients. Accordingly, when viewed as polynomial coefficients, the proposed triangular array produces a family of polynomials that diverge from $(x+1)^n$ for $n \geq 4$. We show that the divergence polynomial must have a root at 1 for every $n\geq 0$. In addition, we define a \emph{log-concavity deviation} measure in order to quantify how far a \emph{near}-log-concave sequence is to being log-concave. We prove that the log-concavity deviation of the triangular progression converges to $\log(4/3)$ under logarithmic scaling. 

Due to the positivity of every element in the triangular progression, we reach infinitely many probability mass functions by normalising the triangular array rows using powers of $1/2$. Accordingly, the construct forms a lower-triangular stochastic matrix of a Markov chain with infinite states. Moreover, we show that a Markov chain governing this stochastic matrix satisfies the three aforementioned state-transition characteristics. A central property of interest in the study of infinite state transitions is the long-term behavior of statistical moments. While expanding arrays typically exhibit unbounded dispersion as the system dimension grows (e.g. symmetric binomial distribution induced by Pascal's triangle), the progression introduced here demonstrates a structural decoupling of a divergent mean and a convergent variance. Accordingly, this work explores how a simple, combinatorial rule can bound variance toward a steady-state limit even as the underlying support expands indefinitely. In this spirit, one intriguing feature of the triangular array is that the probability distribution progressively mirrors a \emph{wave} organizing itself inside a discrete channel as its mean advances (i.e., slides) with increasing $n$. As the variance converges while the relative peak profile remains structurally pinned, the distribution acts as a non-dispersive wave pulse propagating forward through its corridor without flattening out, instead dynamically converging into a stable shape in the form of a limiting probability mass function. This observation also raises the possibility that the three state-transition characteristics that motivated the construction may be related more generally to non-dispersive propagation in discrete stochastic systems; a direction we leave for future investigation.

We also prove that the normalised array hosts an infinite Toeplitz matrix that is totally non-negative of order 2 with each row log-concave. Since every probability mass function extracted from the structure is \emph{near}-log-concave, we define what we call \emph{near}-log-concave random variables and their induced probability measures. Here, we highlight that log-concavity arises abundantly in probability and statistics; e.g. Gaussian, Bernoulli, Poisson, exponential and geometric distributions are log-concave distributions -- see \cite{0b,0aaa,1,6a,6aa,9aaa,13} and references therein. Next, we study the Shannon entropy dynamics -- see \cite{0,0a,2a,5a,7a} -- of the normalised triangular array progression to quantify the evolution of information from one row to the next. We provide a numerical convergence result for the entropy sequence. 

For purposes strictly surrounding the aforementioned state-transition characteristics, having \emph{near}-log-concavity is not a necessary condition, but it leads us to prove unimodality, which is necessary (but not sufficient). More precisely, a sequence $[\eta_1,\ldots,\eta_{n+1}]$ for $n\geq 1$ is unimodal if there exists a $k^* \leq n+1$ such that the following ordering of its elements holds:
\begin{align}
\eta_1 \leq \eta_2 \ldots \leq \eta_{k^*} \,\,\,\,\, \text{and} \,\,\,\,\, \eta_{k^*} \geq \ldots \geq \eta_{n+1}. \label{unimodality}
\end{align}
If each row of a stochastic matrix satisfies (\ref{unimodality}) at $k^* = n$ for any $n\geq 1$, then the first two state-transition characteristics are satisfied -- this will be clear later in the paper, where we show that each row of the normalised array progression satisfies this property.

The purpose of this paper is exploratory: we investigate certain mathematical features of a particular construction and the structures it naturally induces. The results below should therefore be viewed as a collection of properties and consequences of the construct, with the state-transition interpretation providing the original motivation.

\section{Triangular Array Progression}
Let $\T_{n,j}$ be the $n$th row and the $j$th column of the triangular array $\T$ for every $n\geq 0$ and $1 \leq j \leq n+1$, given that $n,j \in \Z_+$. Our focus on \emph{triangular} constructs stems from the third state-transition property mentioned in the previous section -- \emph{from any state, the system can transition only to the closest unvisited state in the next time step} -- for which, a lower-triangular structure naturally embeds this characteristic when normalised into a stochastic-matrix, and when the upper-triangular cut is assumed to be all zeros. In order to capture the first and second state-transition properties given in the previous section, we ask additional properties from the infinite triangular array $\T$. For what follows, let $M^*\in \Z_+$ with $0 < M^* < \infty$ be such that all rows $n \leq M^*$ of $\T$ are initialized and there exists no other initialization of $\T$ for any row $n > M^*$. Hereafter, every initialized row of $\T$ should be understood to satisfy all the three state-transition properties given in the previous section under a given normalization. Accordingly, for every row $n > M^*$, we shall first propose the following general array progression algorithm, which will later produce the special tringular array we will focus on.
\\
\\
\textbf{General Construct of $\T$ after $M^*$:}
\begin{enumerate}
\label{algorithmgeneral}
\item $\T_{n,n+1} = F(\T_{k,k+1} \, : \, 0 \leq k \leq n - 1 )$ for every $n > M^*$
\item $\T_{n,n}   = G(\T_{k,k} \, : \, 0 \leq k \leq n - 1 )$ for every $n > M^*$
\item $\T_{n,j}   = H^{(j)}(\T_{k,j} \, : \, 0 \leq k \leq n - 1 )$ for every $1 \leq j \leq n-1$ and $n > M^*$ 
\end{enumerate}
given that $F,G$ and $H^{(j)}$ for $j\geq 1$ are $\R_+$-valued functionals on sequences that satisfy 
\begin{enumerate}
\label{functionalrelations}
\item $G(\T_{k,k} \, : \, 0 \leq k \leq n - 1 ) \geq F(\T_{k,k+1} \, : \, 0 \leq k \leq n - 1 )$ for every $n > M^*$
\item $G(\T_{k,k} \, : \, 0 \leq k \leq n - 1 ) \geq  H^{(n-1)}(\T_{k,n-1} \, : \, 0 \leq k \leq n - 1 )$ for every $n > M^*$
\item $H^{(j)}(\T_{k,j} \, : \, 0 \leq k \leq n - 1 ) \geq H^{(j-1)}(\T_{k,j} \, : \, 0 \leq k \leq n - 1 )$ for every $2 \leq j \leq n-1$ and $n > M^*$ 
\end{enumerate}
Although the construct of $\T$ given above is fairly abstract for specific applications, it provides a mathematical structure that encapsulates all the three state-transition properties we aim to capture. More precisely, the functional relations above produce unimodal sequences where the mode manifests at the $n$th column of the $n$th row until which the sequence is monotonically non-decreasing. In addition, the $(n+1)$th column is at most as high as the $n$th column, and the triangular structure confines any view beyond column $n+1$. For the rest of this paper, we shall introduce and focus on an explicit construct of $\T$ that can be viewed as a particular example of the progression given above. Accordingly, $\T$ that forms our anchor integer structure is produced through the following special sequence progression algorithm:
\begin{enumerate}
\label{algorithm1}
\item $\T_{n,n+1} = 1$ for every $n \geq 0$
\item $\T_{2,2}$ = 2 and $\T_{3,3} = 3$
\item $\T_{n,n} = \sum_{k=1}^{n-2} \T_{n-k,n-k}$ for every $n \geq 4$
\item $\T_{n,j} = \T_{n-1,j} + \psi(j)$ for every $n \geq 1$ and $1 \leq j \leq n-1$
\end{enumerate}
given that $\psi: \Z_+ \rightarrow \N_0$, where 
$\psi(1) = 0$, $\psi(2) = 1$, $\psi(3) = 2$ and
\begin{align}
&\psi(j) = \sum_{k=1}^{j-1} \psi(k) \hspace{0.1in} \text{for $j \geq 4$}. \label{columnprogressionfunc}
\end{align}
\begin{rem}
The explicit construct of $\T$ above is an example of the general construct when one sets $M^* = 3$ and initialize every row $n \leq 3$ in agreement with the explicit algorithm.
\end{rem}
Using the recurrence relation, we see that the entire array is given by the following representation:
\begin{equation}
\label{algorithmentire}
\T_{n,j} =
\begin{cases}
1 & j = 1 \, \text{or} \, j = n+1 \,\, \text{for} \,\, n\geq 0,\\
n & j = 2 \,\, \text{for} \,\, n\geq 1, \\
2n - 3 & j = 3 \,\, \text{for} \,\, n\geq 2,\\
\left(5 + 3(n-j)\right)2^{j-4} & 4 \leq j \leq n. \\
\end{cases}
\end{equation}
Using the initial values $\psi(1) = 0$, $\psi(2) = 1$, $\psi(3) = 2$ and the initial values $\T_{2,2}$ = 2 and $\T_{3,3} = 3$, we have the following boundary relation, which we shall use later on:
\begin{align}
\T_{n-1,j} + \psi(j) < 2\T_{n-1,j} \,\, \Rightarrow \,\, \frac{\T_{n,j}}{\T_{n-1,j}} < 2, \label{ratiorows}
\end{align}
for every $n \geq 1$ and $1 \leq j \leq n-1$, since $x + c < 2x$ for any $c < x$ for $x,c \in \R_+$. Hereafter, we denote $\T_{n,*}$ as the full sequence of numbers at the $nth$ row of $\T$, given by
\begin{align}
\T_{n,*} \triangleq [\T_{n,1}\,,\, \ldots \,,\, \T_{n,n+1}], \notag
\end{align}
for every $n\geq 0$, where $[\, . \,]$ is the notation we use for any sequence -- in this paper, any $[\, . \,]$ is composed of $\R$-valued elements and any infinite sequence should be understood asymptotically. Note that $\T_{n,*}$ for $n\in\{0,1,2,3\}$ are symmetric sequences and coincide with the first four rows of Pascal's triangle. Accordingly, if we define a family of polynomials by
\begin{align}
\label{triangularpolynomial}
p^{(n)}_{\T}(x) = \sum_{j=1}^{n+1} \T_{n,j} x^{j-1} \hspace{0.1in} \forall x \in \R, 
\end{align}
for $n\geq 0$, then $p^{(n)}_{\T}(x) = (x+1)^n$ for $n\in\{0,1,2,3\}$. Thereafter, $\T_{n,*}$ for every $n\geq 4$ is an asymmetric integer sequence and diverges from binomial coefficients -- we shall prove a result regarding this divergence later in the paper.
\begin{lem}
\label{mainlemmaone}
Define an infinite sub-triangular array $\tau$ by
\begin{align}
\tau_{n,j} \triangleq \T_{n + 3, j + 3} \notag
\end{align}
for every $n \geq 0$ and $1 \leq j \leq n +1$. Then, $\tau_{n,*}$ is a strictly-positive log-concave sequence for every $n\geq 2$.
\end{lem}
\begin{proof}
To begin with, every element in $\T$ is strictly-positive, hence, every element in $\tau$ is strictly-positive. Compute the first four sequences of $\tau$, which yield 
\begin{align}
\tau_{0,*} = [1] \hspace{0.1in} \text{and} \hspace{0.1in} \tau_{1,*} = [5,1] \hspace{0.1in} \text{and} \hspace{0.1in} \tau_{2,*} = [8,10,1] \hspace{0.1in} \text{and} \hspace{0.1in} \tau_{3,*} = [11,16,20,1], \label{initialtauseq} 
\end{align}
for which log-concavity is not relevant for $\tau_{0,*}$ nor $\tau_{1,*}$, whereby $\tau_{2,*}$ and $\tau_{3,*}$ are log-concave sequences, since they satisfy 
\begin{align}
\tau_{2,1}\tau_{2,3} \leq \tau_{2,2}^2 \hspace{0.1in} \text{and} \hspace{0.1in} \tau_{3,1}\tau_{3,3} \leq \tau_{3,2}^2 \hspace{0.1in} \text{and} \hspace{0.1in} \tau_{3,2}\tau_{3,4} \leq \tau_{3,3}^2, \label{initiallogconcaveprop} 
\end{align}
respectively. Define the following continuous functions: 
\begin{align}
f_1(x) = 11 + 3x, \,\,\,\, f_2(x) = 16 + 6x, \,\,\,\, f_3(x) = 20 + 12x, \notag 
\end{align}
for any $x\in\R_+$. Since $\T_{n,j} = \T_{n-1,j} + \psi(j)$ for every $n \geq 1$ and $1 \leq j \leq n-1$, the first three columns of $\tau$ define discrete points on $f_1$, $f_2$ and $f_3$, respectively. Note also that we have the positivity:
\begin{align}
f_2(x)^2 - f_1(x)f_3(x) = 36, \notag 
\end{align}
for any $x\in\R_+$. This implies that any sequence formed from the first three columns of $\tau$ satisfy log-concavity:
\begin{align}
\label{funcapprox}
\tau_{n,1}\tau_{n,3} \leq \tau_{n,2}^2,
\end{align}
for every $n\geq 2$. Next, from (\ref{columnprogressionfunc}), we can write
\begin{equation}
\label{columnprogressiontwo}
\psi(j) =
\begin{cases}
j -1 & \text{for $1 \leq j \leq 4$},\\
2\psi(j-1) &  \text{for $j \geq 5$}. \\
\end{cases}
\end{equation}
In addition, since $\T_{2,2}$ = 2, $\T_{3,3} = 3$ and $\T_{n,n} = \sum_{k=1}^{n-2} \T_{n-k,n-k}$ for every $n \geq 4$, we also have the following diagonal relation:
\begin{align}
\T_{n,n} = 2\T_{n-1,n-1} \label{remakeofdiagonals}
\end{align}
for every $ n \geq 5$. Therefore, using (\ref{initialtauseq}), (\ref{columnprogressiontwo}), (\ref{remakeofdiagonals}) and $\T_{n,j} = \T_{n-1,j} + \psi(j)$ for every $n \geq 1$ and $1 \leq j \leq n-1$, we must have
\begin{align}
\tau_{n,j} = 2 \tau_{n-1,j-1}, \label{maindiagonalrelationfirsteq}
\end{align}
for every $n \geq 2$ and $2 \leq j \leq n$. Moreover, since we have $\T_{n,n+1} = 1$ for every $n \geq 0$, we also have $\tau_{n,n+1} = 1$, which further means $\tau_{n,n+1} = \tau_{n-1,n}$ for every $n \geq 1$. Accordingly, we can write
\begin{align}
\label{maindiagonalrelation}
\tau_{n,j} = \alpha_j(n) \tau_{n-1,j-1}
\end{align}
for every $n \geq 2$ and $2 \leq j \leq n+1$, where we defined the coefficient
\begin{equation}
\alpha_j(n) =
\begin{cases}
1 & \text{for $j = n+1$},\\
2 &  \text{for $2 \leq j \leq n$}. \notag
\end{cases}
\end{equation}
Thus, using (\ref{funcapprox}) and (\ref{maindiagonalrelation}), we must have 
\begin{align}
\tau_{n,j-1}\tau_{n,j+1} \leq \tau_{n,j}^2, \notag 
\end{align}
for every $n\geq 2$ and $2 \leq j \leq n$. Therefore, the sequence $\tau_{n,*}$ is a strictly-positive log-concave sequence for every $n\geq 2$.
\end{proof}
\begin{coro}
\label{unimodalitycorollaryres}
Since each $\tau_{n,*}$ is a strictly-positive log-concave sequence for every $n\geq 2$, each $\tau_{n,*}$ is a unimodal sequence with no zeros for every $n\geq 2$. In addition, $\sup(\tau_{n,*}) = \tau_{n,n}$ for every $n\geq 1$.
\end{coro}
We highlight Corollary \ref{unimodalitycorollaryres} since the expressed unimodality is necessary for the required state-transition characteristics as discussed previously. Using the arguments in the proof of Lemma \ref{mainlemmaone}, we can re-write the algorithm by including $\T_{4,4} = 5$ in step 2, and modifying step 3 as $\T_{n,n} = 2\T_{n-1,n-1}$ for every $n \geq 5$; the advantage of doing this is computational efficiency, if one is to code the structure. 
We shall provide an additional property of the sub-triangular array $\tau$.
\begin{prop}
For any $n\geq 3$, $\tau_{n,n} = \tau_{n+3,n - 2}$.
\end{prop}
\begin{proof}
As an initial condition, $\tau_{3,3} = \tau_{6,1}$ holds, and the statement follows from (\ref{maindiagonalrelation}) for every $n\geq 3$.
\end{proof}
For what follows, we shall define what we call \emph{near}-log-concave sequences of degree $r\in \Z_+$. Hereafter, we shall use $|\mathcal{I}|$ for the cardinality of a set $\mathcal{I}$.
\begin{defn}
\label{nearlogconcave}
A sequence $[\eta_1,\ldots,\eta_{n+1}]$ for $n\geq 2$ is a \emph{near}-log-concave sequence of degree $r\in \Z_+$ for finite $0 \leq r \leq n-2$, if the following holds:
\begin{align}
\eta_{i-1}\eta_{i+1} \leq \eta^2_{i} \hspace{0.15in} \text{$\forall i\in\mathcal{I}_{n}^{(r)}$}, \notag
\end{align}
where $\mathcal{I}_{n}^{(r)} \subseteq \mathcal{I}_{n} = \{2,\ldots,n\}$ such that 
\begin{align}
|\mathcal{I}_{n}^{(r)}| = n - r - 1.
\end{align}
\end{defn}
A \emph{near}-log-concave sequence of degree $n-2$ for any finite $n\geq 2$ implies only one $j^*\in\{2,\ldots,n\}$ such that $\eta_{j^*-1}\eta_{j^*+1} \leq \eta^2_{j^*}$ holds with $\mathcal{I}^{(n-2)} = \{j^*\}$; this is when a \emph{near}-log-concave sequence is furthest from being log-concave in terms of the number of elements that satisfy the criteria. In addition, since Definition \ref{nearlogconcave} encapsulates the lower-boundary $r=0$, a \emph{near}-log-concave sequence of degree 0 is a log-concave sequence with $\mathcal{I}_{n}^{(0)} = \mathcal{I}_{n}$. In the limit $n\rightarrow\infty$, we shall later propose and use a measure of deviation of a \emph{near}-log-concave sequence from being a log-concave sequence. Note that $\T_{2,*}$, $\T_{3,*}$ and $\T_{4,*}$ are \emph{near}-log-concave sequences of degree 0; hence, they are log-concave. Also, note that regarding $\T_{0,*}$ and $\T_{1,*}$, \emph{near}-log-concavity is not relevant.
\begin{prop}
\label{propnearlogconcave}
The sequence $\T_{n,*}$ is strictly-positive \emph{near}-log-concave of degree 1 with $\mathcal{I}_{n}^{(1)} = \mathcal{I}_{n}\setminus\{4\}$ for every $n\geq 5$.
\end{prop}
\begin{proof}
By virtue of Lemma \ref{mainlemmaone}, we only need to consider the first five columns of the triangular array $\T$. First, define the following continuous functions: 
\begin{align}
g_1(x) = 1, \,\,\,\, g_2(x) = 5 + x, \,\,\,\, g_3(x) = 7 + 2x, \notag
\end{align}
for any $x\in\R_+$. Since $\T_{n,j} = \T_{n-1,j} + \psi(j)$ for every $n \geq 1$ and $1 \leq j \leq n-1$, the first three columns of $\T$ for $n\geq 5$ define discrete points on $g_1$, $g_2$ and $g_3$, respectively. Clearly, we have 
\begin{align}
g_2(x)^2 - g_1(x)g_3(x) \geq 0, \notag 
\end{align}
for every $x\geq 0$, which shows the log-concavity of the first three columns of $\T$ for any $n\geq 5$. Similarly, define $g_4(x) = 8 + 3x$, from which we again have the log-concavity property 
\begin{align}
g_3(x)^2 - g_2(x)g_4(x) \geq 0, \notag 
\end{align}
for every $x\geq 0$. Finally, define $g_5(x) = 10 + 6x$ where log-concavity fails with 
\begin{align}
g_4(x)^2 - g_3(x)g_5(x) < 0 \notag 
\end{align}
for every $x > 0$. Thus, using Lemma \ref{mainlemmaone}, the statement follows.
\end{proof}
From Proposition \ref{propnearlogconcave}, $\T$ hosts infinitely many \emph{unimodal} \emph{near}-log-concave sequences with no zeros. We also see that log-concavity always breaks at the same column of the array at $j=4$. Accordingly,
\begin{align}
\left( \T^2_{n,i} - \T_{n,i-1}\T_{n,i+1} \right)\left( \T^2_{n+k,i} - \T_{n+k,i-1}\T_{n+k,i+1} \right) \geq 0 \hspace{0.1in} \text{$\forall i=2,\ldots,n$,} \notag
\end{align}
for every $n\geq 5$ for any $k\geq 0$. In addition, from (\ref{ratiorows}) and (\ref{remakeofdiagonals}), we have $\sup(\T_{n,*}) = \T_{n,n}$. Next, we define a \emph{log-concavity deviation} measure to quantify how far a \emph{near}-log-concave sequence is to being a log-concave sequence.
\begin{defn}
Let $\eta = [\eta_1,\ldots,\eta_{n+1}]$ for $n\geq 2$ be a sequence, where $\min(\eta) > 0$. Define 
\begin{align}
\Theta_{i}(\eta) = \max\left( \log(\eta_{i-1}) + \log(\eta_{i+1}) - 2\log(\eta_i) \,,\, 0 \right), \label{differentialmetricfunction}
\end{align}
for every $i\in\{2,\ldots,n\}$. Then,
\begin{align}
\Theta(\eta) = \sum_{i=2}^{n} \Theta_{i}(\eta), \label{differentialmetricfunctionsummation}
\end{align}
is called the \emph{log-concavity deviation} of $\eta$ under logarithmic scaling.
\end{defn}
\begin{prop}
For any sequence $\eta$ where $\min(\eta) > 0$, $\Theta(\eta) = 0$ if and only if $\eta$ is log-concave.
\end{prop}
\begin{proof}
Let $\eta$ be a sequence such that $\min(\eta) > 0$. If $\eta$ is log-concave, then $\Theta(\eta) = 0$.  For any $\eta$ that is not log-concave, there exists at least one $j^*\in\{2,\ldots,|\eta|-1\}$ such that $\Theta_{j^*}(\eta) > 0$, using (\ref{differentialmetricfunction}). The statement follows from (\ref{differentialmetricfunctionsummation}). 
\end{proof}
%We shall now prove an asymptotic behaviour of  $\T$ in terms of the log-concavity deviation.
\begin{prop}
\label{logconcavecorrection}
Let $\Theta$ be as in (\ref{differentialmetricfunctionsummation}). Then, $\Theta(\T_{n,*}) \rightarrow \log(4/3)$ as $n \rightarrow \infty$.
\end{prop}
\begin{proof}
Using Proposition \ref{propnearlogconcave}, we have $\Theta_{i}\left( \T_{n,*} \right) = 0$ for every $i \in \mathcal{I}_{n}^{(1)}$, for every $n\geq 5$, which implies the following:
\begin{align}
\sum_{i \in \mathcal{I}_{n}^{(1)}} \Theta_{i}( \T_{n,*} ) = 0 \hspace{0.1in} \text{$\forall n\geq 5$} \hspace{0.1in} \Rightarrow \lim_{n \rightarrow \infty}\Theta(\T_{n,*}) = \lim_{n \rightarrow \infty} \Theta_{4}(\T_{n,*}). \label{allzeroselsewhere}
\end{align}
From Proposition \ref{propnearlogconcave}, recall the functions $g_3(x) = 7 + 2x$, $g_4(x) = 8 + 3x$  and $g_5(x) = 10 + 6x$ for any $x\in\R_+$. Accordingly,
\begin{align}
\lim_{n \rightarrow \infty} \left(\log(\T_{n,3}) + \log(\T_{n,5}) - 2\log(\T_{n,4})\right) = \lim_{x \rightarrow \infty} \left(\log(g_3(x)) + \log(g_5(x)) - 2\log(g_4(x))\right). \label{mainlimitconvergencenum}
\end{align}
Since $\log(\cdot)$ is strictly continuous on the domain $(0, \infty)$, we have the following limit:
\begin{align*}
\lim_{x \rightarrow \infty} \left(\log(g_3(x)) + \log(g_5(x)) - 2\log(g_4(x))\right) &= \lim_{x \rightarrow \infty} \log \left( \frac{g_3(x)g_5(x)}{g_4(x)^2} \right) = \log \left( \lim_{x \rightarrow \infty} \frac{12x^2 + 62x + 70}{9x^2 + 48x + 64} \right). \notag
\end{align*}
Therefore, from (\ref{mainlimitconvergencenum}), it follows that
\begin{align}
\lim_{n \rightarrow \infty} \left(\log(\T_{n,3}) + \log(\T_{n,5}) - 2\log(\T_{n,4})\right) = \log \left( \lim_{x \rightarrow \infty} \frac{12 + \frac{62}{x} + \frac{70}{x^2}}{9 + \frac{48}{x} + \frac{64}{x^2}} \right) = \log \left( \frac{4}{3} \right). \label{mainlimitconvergencenumfinal}
\end{align}
Hence, using (\ref{allzeroselsewhere}), (\ref{mainlimitconvergencenumfinal}) and the continuity of the $\max$ function, we have the convergence $\Theta(\T_{n,*}) \rightarrow \log(4/3)$ as $n \rightarrow \infty$.
\end{proof}
One can interpret Proposition \ref{propnearlogconcave} and Proposition \ref{logconcavecorrection} together as follows: $\T$ starts log-concave, becomes \emph{near}-log-concave of degree 1 soon after, and its log-concavity deviation progressively converges to a constant above zero under logarithmic scaling. We shall prove another asymptotic behavior of $\T$. Define a sequence given by the product of each row as follows: 
\begin{align}
S_n = \prod_{j=1}^{n+1} \T_{n,j}, \notag 
\end{align}
for every $n\geq 0$. Next, define the following ratio:
\begin{align}
\phi_n = \frac{S_{n-1}S_{n+1}}{(S_n)^2} \label{producsequence}
\end{align}
for every $n\geq 1$. We are now in position to state the following result.
\begin{prop}
\label{asymptoticproposition}
Let $\phi_n$ be as in (\ref{producsequence}) for every $n\geq 1$. Then, $\phi_n \rightarrow 2$ as $n \rightarrow \infty$.
\end{prop}
\begin{proof}
Define the following triplets from sub-sequences of $\T$ in the following way:
\begin{align}
s^{(4)}_{n-1} = \prod_{j=4}^{n} \T_{n-1,j} \hspace{0.1in} \text{and} \hspace{0.1in} s^{(5)}_{n} = \prod_{j=5}^{n+1} \T_{n,j} \hspace{0.1in} \text{and} \hspace{0.1in} s^{(6)}_{n+1} = \prod_{j=6}^{n+2} \T_{n+1,j}\notag
\end{align}
for $n\geq 6$. Using (\ref{maindiagonalrelation}) from Lemma \ref{mainlemmaone}, we must have the log-concavity ratio satisfy the following relation:
\begin{align}
\frac{s^{(4)}_{n-1}s^{(6)}_{n+1}}{(s^{(5)}_n)^2} = 1, \label{asymptitcone} 
\end{align}
for every $n\geq 6$. Hence, it suffices to study the product asymptotes of the first 6 columns of $\T$ for $n\geq 6$. Accordingly, define the following functions:
\begin{align}
&W_4(k) = (5+k)(7+2k) \notag \\
&W_5(k) = (5+k+1)(7+2(k+1))(8+3(k+1))\notag \\
&W_6(k) = (5+k+2)(7+2(k+2))(8+3(k+2))(16+6(k+1)) \notag
\end{align}
for every $k\in\Z_+$. Note that  
\begin{align}
W_4(k) = \prod_{j=1}^{3} \T_{k+5,j} \hspace{0.1in} \text{and} \hspace{0.1in} W_5(k) = \prod_{j=1}^{4} \T_{k+6,j} \hspace{0.1in} \text{and} \hspace{0.1in} W_6(k) = \prod_{j=1}^{5} \T_{k+7,j}, \notag
\end{align}
for every $k\in\Z_+$. Accordingly, define the ratio:
\begin{align}
W^{(*)}(k) = \frac{W_4(k)W_6(k)}{(W_5(k))^2} \hspace{0.1in} \forall k\in\Z_+.
\end{align}
Thus, the Laurent series expansion of $W^{(*)}$ at $k=\infty$, which we denote as $\mathbb{L}_{\infty}(W^{(*)})$, is
\begin{align}
\mathbb{L}_{\infty}(W^{(*)}) = 2 + \frac{2}{k} - \frac{34}{3k^2} + \frac{584}{9k^3} + \text{\emph{O}}\left( \left(\frac{1}{k}\right)^4 \right), \notag
\end{align}
Hence, $W^{(*)}(k) \rightarrow 2$ as $k\rightarrow \infty$. Finally, note that
\begin{align}
S_{n-1} = W_4(n-6)s^{(4)}_{n-1} \hspace{0.1in} \text{and} \hspace{0.1in} S_{n} = W_5(n-6)s^{(5)}_{n} \hspace{0.1in} \text{and} \hspace{0.1in} S_{n+1} = W_6(n-6)s^{(6)}_{n+1}, \notag
\end{align}
for every $n\geq 6$, which proves the statement using (\ref{asymptitcone}).
\end{proof}
Note that the ratio in (\ref{producsequence}) also quantifies a form of log-concavity deviation -- Proposition \ref{asymptoticproposition} shows that a sequence formed by the products $\{S_n \, : \, n\geq 1\}$ does not achieve log-concavity asymptotically. We shall now continue our analysis through the sum of each row of $\T$, which brings forward an elegant property. For what follows, we denote
\begin{align}
||\T_{n,*}|| \triangleq \sum_{j=1}^{n+1} \T_{n,j}, \notag
\end{align}
which is the $\mathcal{L}_1$-norm of the sequence, since every element is positive.
\begin{prop}
\label{summationpowertwo}
For every $n\geq 0$, $||\T_{n,*}|| = 2^{n}$.
\end{prop}
\begin{proof}
First five rows of $\T$ can be quickly checked with 
\begin{align}
||\T_{0,*}|| = 1, \,\,\,\, ||\T_{1,*}|| = 2, \,\,\,\, ||\T_{2,*}|| = 4, \,\,\,\, ||\T_{3,*}|| = 8, \,\,\,\, ||\T_{4,*}|| = 16. \notag 
\end{align}
Hence, $||\T_{n,*}|| = 2^{n}$ for $0 \leq n \leq 4$. Next, let $\tau_{n,(i\rightarrow k)}$ be a sub-sequence of $\tau_{n,*}$ given by
\begin{align}
\tau_{n,(i\rightarrow k)} \triangleq [\tau_{n,i},\ldots,\tau_{n,k}] \notag
\end{align}
for any $1 \leq i \leq k \leq n + 1$. From Lemma \ref{mainlemmaone}, we know $\tau_{n,j} = 2 \tau_{n-1,j-1}$ for every $n \geq 2$ and $2 \leq j \leq n$, hence, we must have
\begin{align}
||\tau_{n+1,(2\rightarrow n+1)}|| = 2||\tau_{n,(1\rightarrow n)}|| \label{diagonalsummationrel}
\end{align}
for every $n \geq 2$. Accordingly, since $\T_{n,n+1} = 1$ for every $n\geq 0$ and $\T_{n,1} = 1$ from $\T_{n,1} = \T_{n-1,1} + \psi(1)$ for every $n \geq 1$ and $1 \leq j \leq n-1$, it remains to prove
\begin{align}
||\T_{n+1,(2\rightarrow 4)}|| = 2 ||\T_{n,(2\rightarrow 3)}|| + 2 \label{diagonalsummationreltwo}
\end{align}
for $n\geq 5$, since (\ref{diagonalsummationrel}) holds. As done in Proposition \ref{propnearlogconcave}, let $g_2(x) = 5 + x$, $g_3(x) = 7 + 2x$ and $g_4(x) = 8 + 3x$. Now define $h_1,h_2 : \N_+ \rightarrow \N_+$ as follows:
\begin{align}
&h_1(k) = g_2(k) + g_3(k) + 2 = 14 + 3k \notag \\
&h_2(k) = g_2(k+1) + g_3(k+1) + g_4(k+1) + 2 = 22 + 6(k+1), \notag
\end{align}
for every integer $k\geq 0$, which correspond to the summations $||\T_{n,(2\rightarrow 3)}|| + 2$ and $||\T_{n+1,(2\rightarrow 4)}|| + 2$, respectively, for every $n\geq 5$. Thus, since we have $h_2(k) / h_1(k) = 2$ for every $k\geq 0$, we must also have
\begin{align}
\frac{||\T_{n+1,(2\rightarrow 4)}|| + 2}{||\T_{n,(2\rightarrow 3)}|| + 2} = 2, \notag
\end{align}
which proves (\ref{diagonalsummationreltwo}). Therefore, using (\ref{diagonalsummationrel}) and (\ref{diagonalsummationreltwo}), we must have $||\T_{n,*}|| = 2^{n}$ for every $n\geq 5$. The statement follows since we also know $||\T_{n,*}|| = 2^{n}$ for $0 \leq n \leq 4$.
\end{proof}
Every element in $\T_{n,*}$ is unique for $n\geq 5$, except for $\T_{n,1} = \T_{n,n+1} = 1$. Hence, using Proposition \ref{summationpowertwo}, we can also write 
\begin{align}
||\T_{n,*}|| = 2^{|\T_{n,*}|}, \notag
\end{align} 
for $n\geq 5$, where $|\T_{n,*}|$ denotes the cardinality of the set of elements in $\T_{n,*}$. Note also that we have
\begin{align}
\sum_{n=0}^N\sum_{j=1}^{n+1} \T_{n,j} = \sum_{n=0}^N ||\T_{n,*}|| = 2^{N+1} - 1, \notag
\end{align}
for any finite $N\geq 0$, as the sum of all integer elements of every row until the $N$th row (including), which follows from Proposition \ref{summationpowertwo}.
\begin{coro}
For every $n\geq 1$, $\T_{n,*}$ defines the degree sequence of a multigraph (without loops).
\end{coro}
\begin{proof}
Every element of $\T_{n,*}$ is positive and $||\T_{n,*}||$ is an even number for every $n\geq 1$ from Proposition \ref{summationpowertwo}. In addition, since we have $\sup(\T_{n,*}) = \T_{n,n}$, we need to show
\begin{align}
\T_{n,n} \leq \sum_{j\neq n} \T_{n,j}, \label{maximumboundgraph}
\end{align}
for every $n\geq 1$. This can be verified by direct computation for $n\in\{1,2\}$. As for $3 \leq n \leq 6$, it can be observed that $\T_{n,n} \leq \T_{n,n-2} + \T_{n,n-1}$. Finally, for any $n>6$, we must also have 
\begin{align}
\T_{n,n} \leq \T_{n,n-2} + \T_{n,n-1}, \notag 
\end{align}
due to (\ref{maindiagonalrelationfirsteq}). Hence, (\ref{maximumboundgraph}) holds for every $n\geq 1$, and the statement follows from \cite{graphone}.
\end{proof}
Proposition \ref{summationpowertwo} further allows us to study how polynomials defined on $\T$ deviate from those defined on Pascal's triangle (A007318 in OEIS) -- similar results can be explored for other sequences.
\begin{coro}
\label{polynomialcorollary}
Let $p^{(n)}_{\T}$ be the polynomial as given in (\ref{triangularpolynomial}) for every $n\geq 0$. Define 
\begin{align}
q^{(n)}(x) = p^{(n)}_{\T}(x) - (x+1)^n \hspace{0.1in} \forall x \in \R, \notag
\end{align}
as the deviation polynomial $q^{(n)}$ for every $n\geq 0$. Then, 
\begin{enumerate}
\item $q^{(n)} = 0$ for $n\in\{0,1,2,3\}$
\item $q^{(n)}$ has degree $n-1$ for $n\geq4$ and its coefficients sum to zero
\end{enumerate}
Accordingly, $q^{(n)}$ has a root at $x=1$ for every $n\geq 0$.
\end{coro}
\begin{proof}
Recall that $(x+1)^n$ has binomial coefficients for every $n\geq0$. The first statement is trivial by direct calculation of $\T_{n,*}$ for $n\in\{0,1,2,3\}$ and comparing the sequences with the first four rows of Pascal's triangle. For the second statement, since  $\T_{2,2}$ = 2 and $\T_{n,j} = \T_{n-1,j} + \psi(j)$ for every $n \geq 1$ and $1 \leq j \leq n-1$, where $\psi(1) = 0$ and $\psi(2) = 1$, the first two columns of $\T$ must equal the first two columns of Pascal's triangle for every $n\geq 0$. Therefore, the degree of $q^{(n)}$ must be $n-1$ for $n\geq4$. The second part follows directly from Proposition \ref{summationpowertwo} since any row of Pascal's triangle sums to $2^n$ for every $n\geq 0$, hence, the coefficients of $q^{(n)}$ must sum to zero. The final statement follows immediately.
\end{proof}
\begin{rem}
The first 5 rows of our proposed triangular array shares the same values with sequence A046688 in OEIS; an array wherein the $n$th row is an arithmetic progression of difference $2^{n-1}$ for $n > 0$. 
\end{rem}
\begin{rem}
The first 18 values of our proposed sequence shares the same values with the sequence formed by the 2nd to 19th values of A208342 in OEIS.
\end{rem}
Since we have an explicit formula for $||\T_{n,*}||$ for every $n\geq 0$ from Proposition \ref{summationpowertwo}, we can scale $\T$ to achieve a new triangular array $\T^{(y)}$ with a chosen sum progression, through
\begin{align}
\T^{(y)}_{n,j} \triangleq \left(\frac{y}{2}\right)^n\T_{n,j} \label{scaledarray}
\end{align}
for any finite $y\in\R_+\setminus\{0\}$ and every $n\geq 0$ and $1 \leq j \leq n+1$.
\begin{coro}
\label{scaledarrayprogressopn}
Let $\T^{(y)}$ be defined as in (\ref{scaledarray}) for any finite $y\in\R_+\setminus\{0\}$. Then, 
\begin{enumerate}
\item $\T^{(y)}_{n,*}$ is strictly-positive \emph{near}-log-concave of degree 1 with $\mathcal{I}_{n}^{(1)} = \mathcal{I}_{n}\setminus\{4\}$ for every $n\geq 5$ \item $||\T^{(y)}_{n,*}|| = y^{n}$ for every $n\geq 0$
\end{enumerate}
\end{coro}
\begin{proof}
The statement follows from Proposition \ref{propnearlogconcave} and Proposition \ref{summationpowertwo} for finite $y\in\R_+\setminus\{0\}$.
\end{proof}
If we choose $y=2$, we have $\T^{(2)} = \T$. If we choose $y=1$, then Corollary \ref{scaledarrayprogressopn} leads us to a probabilistic setup, where each $\T^{(1)}_{n,*}$ defines a probability mass function for every $n\geq 0$, since all elements of $\T^{(1)}_{n,*}$ are strictly-positive and $||\T^{(1)}_{n,*}|| = 1$. 
\begin{prop}
\label{propforTprobproperties}
The following properties hold for $\T^{(1)}$:
\begin{enumerate}
\item For every $n\geq 5$ and any $k\geq 0$, $\sup(\T^{(1)}_{n,*}) = \sup(\T^{(1)}_{n+k,*})$.
\item For every $j\geq 2$ and any $k\geq 1$, $\sup([\T^{(1)}_{j-1,j},\ldots,\T^{(1)}_{j-1+k,j}]) = \T^{(1)}_{j,j}$.
\end{enumerate}
\end{prop}
\begin{proof}
Using (\ref{scaledarray}) with $y=1$ and (\ref{maindiagonalrelation}), we have $\T^{(1)}_{n-1,j-1} = \T^{(1)}_{n,j}$ for every $n\geq 5$ and $5 \leq j \leq n$. Hence, having $\sup(\T^{(1)}_{n,*})=\T^{(1)}_{n,n}$ for every $n\geq 5$, the first statement follows. The second statement follows from (\ref{ratiorows}) and (\ref{scaledarray}) with $y=1$.
\end{proof}
Any $\T^{(y)}$ for any finite $y\in\R_+\setminus\{0\}$ can be viewed as an infinite-dimensional lower-triangular matrix. An algebraic study of $\T^{(y)}$ is beyond the scope of this paper, but we shall highlight the following. 
\begin{rem}
\label{transitionmatrix}
Note that $\T^{(1)}$ defines a lower-triangular right-stochastic matrix and provides the transition probabilities of a Markov chain with infinite states.
\end{rem}
\begin{prop}
Let $\tau^{(1)}$ be defined by $\tau^{(1)}_{n,j} \triangleq \T^{(1)}_{n + 3, j + 3}$ for every $n \geq 0$ and $1 \leq j \leq n +1$. Then, 
\begin{enumerate}
\item $\tau^{(1)}_{n,*}$ is a strictly-positive log-concave sequence for every $n\geq 2$
\item $\tau^{(1)}$ generates a lower-triangular Toeplitz matrix by setting $\emph{diag}(\tau^{(1)}) = 0$  
\end{enumerate}
\end{prop}
\begin{proof}
The first statement follows from Lemma \ref{mainlemmaone} and (\ref{scaledarray}) with $y=1$. The second statement follows from (\ref{scaledarray}) and (\ref{maindiagonalrelation}).
\end{proof}
Hence, $\T^{(1)}$ hosts an infinite Toeplitz matrix, where each row is a log-concave sequence. Note also that $\tau^{(1)}$, when its diagonals are set to zero, is the Toeplitz matrix of the column sequence $[\tau^{(1)}_{1,1},\tau^{(1)}_{2,1},\ldots]$. 
\begin{rem}
Note that $[\tau^{(1)}_{1,1},\tau^{(1)}_{2,1},\ldots]$ is itself a log-concave sequence, hence, it is a Pólya frequency sequence of order 2, which means the given Toeplitz matrix must be totally non-negative of order 2. 
\end{rem}
We shall now generalise Proposition \ref{logconcavecorrection} for any finite $y\in\R_+\setminus\{0\}$.
\begin{prop}
\label{logconcavecorrectionnormalised}
Let $\Theta$ be as in (\ref{differentialmetricfunctionsummation}). Then, $\Theta(\T^{(y)}_{n,*}) \rightarrow \log(4/3)$ as $n \rightarrow \infty$ for any finite $y\in\R_+\setminus\{0\}$.
\end{prop}
\begin{proof}
Using Corollary \ref{scaledarrayprogressopn}, we have $\Theta_{i}\left( \T^{(y)}_{n,*} \right) = 0$ for every $i \in \mathcal{I}_{n}^{(1)}$, for every $n\geq 5$, which implies
\begin{align}
\sum_{i \in \mathcal{I}_{n}^{(1)}} \Theta_{i}( \T^{(y)}_{n,*} ) = 0 \hspace{0.1in} \text{$\forall n\geq 5$} \hspace{0.1in} \Rightarrow \lim_{n \rightarrow \infty}\Theta(\T^{(y)}_{n,*}) = \lim_{n \rightarrow \infty} \Theta_{4}(\T^{(y)}_{n,*}), \label{allzeroselsewherenormalised}
\end{align}
for any finite $y\in\R_+\setminus\{0\}$. Define $g_3$, $g_4(x)$  and $g_5$ as in Proposition \ref{logconcavecorrection}, so that
\begin{align}
\lim_{n \rightarrow \infty} \left(\log(\T^{(y)}_{n,3}) + \log(\T^{(y)}_{n,5}) - 2\log(\T^{(y)}_{n,4})\right) = \lim_{n \rightarrow \infty} \left(\log(\T_{n,3}) + \log(\T_{n,5}) - 2\log(\T_{n,4})\right), \label{mainlimitconvergencenumnormalised}
\end{align}
since all $(y/2)^n > 0$ terms in (\ref{mainlimitconvergencenumnormalised}) cancel out simply due to $\log(cr) = \log(c) + \log(r)$ for any $c > 0$ and $r>0$. The statement thus follows from Proposition \ref{logconcavecorrection}.
\end{proof}
From Proposition \ref{logconcavecorrectionnormalised}, we see that the log-concavity deviation of $\T^{(y)}$ converges to $\log(4/3)$ asymptotically, for any finite $y\in\R_+\setminus\{0\}$ -- as such Proposition \ref{logconcavecorrection} can be viewed as a corollary to Proposition \ref{logconcavecorrectionnormalised} with $\T^{(2)} = \T$. In addition, similar to Proposition \ref{asymptoticproposition}, define a sequence given by the product: 
\begin{align}
S^{(y)}_n = \prod_{j=1}^{n+1} \T^{(y)}_{n,j} \notag 
\end{align}
for every $n\geq 0$ and any finite $y\in\R_+\setminus\{0\}$, and generate a sequence of ratios:
\begin{align}
\phi^{(y)}_n = \frac{S^{(y)}_{n-1}S^{(y)}_{n+1}}{(S^{(y)}_n)^2} \label{producsequencegeneral}
\end{align}
for every $n\geq 1$. We can now generalise Proposition \ref{asymptoticproposition} as follows.
\begin{prop}
\label{asymptitcgeneraltvalue}
Let $\phi^{(y)}_n$ be as in (\ref{producsequencegeneral}) for every $n\geq 1$. Then, for any finite $y\in\R_+\setminus\{0\}$,
\begin{align}
\phi^{(y)}_n \rightarrow \frac{y^2}{2} \,\,\, \text{as} \,\,\, n \rightarrow \infty. \notag
\end{align}
\end{prop}
\begin{proof}
As in Proposition \ref{asymptoticproposition}, define the following triplets from $\T^{(y)}$:
\begin{align}
s^{(y,4)}_{n-1} = \prod_{j=4}^{n} \left(\frac{y}{2} \right)^{n-1}\T_{n-1,j} \hspace{0.1in} \text{and} \hspace{0.1in} s^{(y,5)}_{n} = \prod_{j=5}^{n+1} \left(\frac{y}{2} \right)^n\T_{n,j} \hspace{0.1in} \text{and} \hspace{0.1in} s^{(y,6)}_{n+1} = \prod_{j=6}^{n+2} \left(\frac{y}{2} \right)^{n+1}\T_{n+1,j}\notag
\end{align}
for $n\geq 6$ and any finite $y\in\R_+\setminus\{0\}$. Using (\ref{maindiagonalrelation}) from Lemma \ref{mainlemmaone}, we must have $(s^{(y,4)}_{n-1}s^{(y,6)}_{n+1})/(s^{(y,5)}_n)^2 = 1$ for every $n\geq 6$ and finite $y\in\R_+\setminus\{0\}$. Thus, similar to Proposition \ref{asymptoticproposition}, define scaled functions as follows:
\begin{align}
&W^{(y)}_4(k) = \left(\frac{y}{2} \right)^{k+5}(5+k)\left(\frac{y}{2} \right)^{k+5}(7+2k)\left(\frac{y}{2} \right)^{k+5}  \notag \\
&W^{(y)}_5(k) = \left(\frac{y}{2} \right)^{k+6}(5+k+1)\left(\frac{y}{2} \right)^{k+6}(7+2(k+1))\left(\frac{y}{2} \right)^{k+6}(8+3(k+1))\left(\frac{y}{2} \right)^{k+6}\notag \\
&W^{(y)}_6(k) = \left(\frac{y}{2} \right)^{k+7}(5+k+2)\left(\frac{y}{2} \right)^{k+7}(7+2(k+2))\left(\frac{y}{2} \right)^{k+7}(8+3(k+2))\left(\frac{y}{2} \right)^{k+7}(16+6(k+1))\left(\frac{y}{2} \right)^{k+7} \notag
\end{align}
for every $k\in\Z_+$, so that
\begin{align}
W^{(y)}_4(k) = \prod_{j=1}^{3} \left(\frac{y}{2} \right)^{k+5}\T_{k+5,j} \hspace{0.1in} \text{and} \hspace{0.1in} W^{(y)}_5(k) = \prod_{j=1}^{4} \left(\frac{y}{2} \right)^{k+6}\T_{k+6,j} \hspace{0.1in} \text{and} \hspace{0.1in} W^{(y)}_6(k) = \prod_{j=1}^{5} \left(\frac{y}{2} \right)^{k+7}\T_{k+7,j}, \notag
\end{align}
for every $k\in\Z_+$. Since we have
\begin{align}
\frac{\left(\frac{y}{2}\right)^{3(k+5)}\left(\frac{y}{2}\right)^{5(k+7)}}{\left(\frac{y}{2}\right)^{8(k+6))}} = \frac{y^2}{4}, \notag
\end{align}
we must have the following:
\begin{align}
W^{(y,*)}(k) = \frac{W^{(y)}_4(k)W^{(y)}_6(k)}{(W^{(y)}_5(k))^2} = \left(\frac{y^2}{4}\right)\frac{W_4(k)W_6(k)}{(W_5(k))^2} = \left( \frac{y^2}{4}\right)W^{(*)}(k) \hspace{0.1in} \forall k\in\Z_+,
\end{align}
for any finite $y\in\R_+\setminus\{0\}$. From Proposition \ref{asymptoticproposition}, we know $W^{(*)}(k) \rightarrow 2$ as $k\rightarrow \infty$, hence, $W^{(y,*)}(k) \rightarrow y^2/2$ as $k\rightarrow \infty$. Finally, since
\begin{align}
S^{(y)}_{n-1} = W^{(y)}_4(n-6)s^{(y,4)}_{n-1} \hspace{0.1in} \text{and} \hspace{0.1in} S^{(y)}_{n} = W^{(y)}_5(n-6)s^{(y,5)}_{n} \hspace{0.1in} \text{and} \hspace{0.1in} S^{(y)}_{n+1} = W^{(y)}_6(n-6)s^{(y,6)}_{n+1}, \notag
\end{align}
for every $n\geq 6$, the statement follows.
\end{proof}
Since $\T^{(2)} = \T$, we have $\phi^{(2)}_n = \phi_n \rightarrow 2$ as $n\rightarrow \infty$. In addition, from Proposition \ref{asymptitcgeneraltvalue}, we can see that for the probabilistic construct $\T^{(1)}$ we have $\phi^{(1)}_n \rightarrow 1/2$ as $n\rightarrow \infty$. 
\begin{rem}
For any $0 < y < \sqrt{2}$, the sequence formed by $\{S^{(y)}_n:n\geq 0\}$ achieves log-concavity asymptotically. This follows since $\T^{(\sqrt{2})}$, with $y=\sqrt{2}$, produces the asymptote 
\begin{align}
\phi^{(\sqrt{2})}_n \rightarrow 1 \,\,\, \text{as} \,\,\, n\rightarrow \infty, \notag
\end{align}
from above, converging to the boundary of log-convexity.
\end{rem}
Hereafter, we shall focus on $\T^{(1)}$ and work on a probability space $(\Omega,\F,\P)$. We shall now introduce a new family of random variables.
\begin{defn}
\label{nearlogconcaverandomvariable}
Let $\mathbf{D}$ be an ordered finite subset of $\R$ with cardinality $|\mathbf{D}|\geq 3$, and $X : \Omega \rightarrow \mathbf{D}$ be a random variable with probability mass function 
\begin{align}
\P\left(X = d_j\right) = P_{X}(d_j), \notag 
\end{align}
for every $d_j \in \mathbf{D}$. The random variable $X$ is called \emph{near}-log-concave of degree $r\in \Z_+$ for finite $0 \leq r \leq |\mathbf{D}|-3$, if 
\begin{align} 
P_{X}(d_{j-1})P_{X}(d_{j+1}) \leq P_{X}(d_j)^2 \hspace{0.15in} \text{$\forall j\in\mathcal{I}_{|\mathbf{D}|}^{(r)}$}, \notag
\end{align}
where $\mathcal{I}_{|\mathbf{D}|}^{(r)} \subseteq \mathcal{I}_{|\mathbf{D}|} = \{2,\ldots,|\mathbf{D}|\}$ such that 
\begin{align}
|\mathcal{I}_{|\mathbf{D}|}^{(r)}| = |\mathbf{D}| - r - 2. \notag
\end{align}
\end{defn}
From Definition \ref{nearlogconcaverandomvariable}, we see that a \emph{near}-log-concave random variable of degree 0 is a log-concave random variable of \cite{1,6a}. 
%Using log-concavity deviation measures, one can quantify how far a \emph{near}-log-concave random variable is from being a log-concave random variable. 
\begin{prop}
\label{nearlogconcaverandomvariableprop}
Let $\mathbf{D}^{(n)}$ be an ordered finite subset of $\R$ with cardinality $|\mathbf{D}^{(n)}| = n+1$ for any $n \geq 2$, and $X^{(n)} : \Omega \rightarrow \mathbf{D}^{(n)}$ be a random variable with probability mass function
\begin{align}
\P\left(X^{(n)} = d^{(n)}_j\right) = \T^{(1)}_{n,j} \notag
\end{align}
for every $d^{(n)}_j \in \mathbf{D}^{(n)}$. Then, the following hold:
\begin{enumerate}
\item $X^{(2)}$, $X^{(3)}$ and $X^{(4)}$ are log-concave random variables with $\mathcal{I}_{|\mathbf{D}^{(n)}|}^{(0)} = \mathcal{I}_{|\mathbf{D}^{(n)}|}$ for $n=2,3,4$
\item $X^{(n)}$ is a \emph{near}-log-concave random variable of degree \emph{1} with $\mathcal{I}_{|\mathbf{D}^{(n)}|}^{(1)} = \mathcal{I}_{|\mathbf{D}^{(n)}|}\setminus\{4\}$ for every $n\geq 5$
\end{enumerate}
given that $\mathcal{I}_{|\mathbf{D}^{(n)}|} = \{2,\ldots,|\mathbf{D}^{(n)}|-1\}$ for every finite $n\geq 2$.
\end{prop}
\begin{proof}
The first statement can be shown by direct calculation. The second statement follows from Corollary \ref{scaledarrayprogressopn} and Definition \ref{nearlogconcaverandomvariable}.
\end{proof}
\begin{rem}
\label{probmeasure}
In Proposition \ref{nearlogconcaverandomvariableprop}, each $\mathbf{D}^{(n)}$ is equipped with its discrete $\sigma$-algebra $\mathcal{E}^{(n)} = \mathcal{P}(\mathbf{D}^{(n)})$, and $X^{(n)}$ is $\F/\mathcal{E}^{(n)}$-measurable. The values $\T^{(1)}_{n,*}$ uniquely define the push-forward probability measure $\mu^{(n)} = \P \circ (X^{(n)})^{-1}$ on $(\mathbf{D}^{(n)},\mathcal{E}^{(n)})$ for each $n \geq 2$, such that $\mu^{(n)}(\{d^{(n)}_j\}) = \T^{(1)}_{n,j}$. Accordingly, $\mu^{(n)}$ is a near-log-concave probability measure for every $n\geq 5$.
\end{rem}
Following the setup given in Proposition \ref{nearlogconcaverandomvariableprop}, and denoting $\E^{\P}[\cdot]$ as the expectation under the probability measure $\P$, the $m$-th moment of $X^{(n)}$, denoted as $\langle\textbf{M}\rangle_{\mathbf{D}^{(n)}}^{(m)}$, is given by
\begin{align} 
\langle\textbf{M}\rangle_{\mathbf{D}^{(n)}}^{(m)} = \sum_{d^{(n)}_j \in \, \mathbf{D}^{(n)}} (d^{(n)}_j)^m \mu^{(n)}\left( \{d^{(n)}_j\} \right) = \E^{\P}\left[ (X^{(n)})^m \right], \notag  \end{align}
for any $n\geq 2$ and $m\geq 1$. 
\begin{rem}
To study the joint behavior of these random variables under the assumption of mutual independence, one can construct an infinite product space $(\mathbf{D}^*, \mathcal{E}^*, \mu^*)$ given by 
\begin{align}
&\mathbf{D}^* = \prod_{n=2}^{\infty} \mathbf{D}^{(n)} = \mathbf{D}^{(2)} \times \mathbf{D}^{(3)} \times \ldots \notag \\
&\mathcal{E}^* = \bigotimes_{n=2}^{\infty} \mathcal{E}^{(n)} = \mathcal{E}^{(2)} \otimes \mathcal{E}^{(3)} \otimes \ldots \notag \\
&\mu^* = \bigotimes_{n=2}^{\infty} \mu^{(n)} = \mu^{(2)} \otimes \mu^{(3)} \otimes \ldots \notag
\end{align}
and define a new sequence of mutually independent random variables $Y^{(n)}: \mathbf{D}^* \rightarrow \mathbf{D}^{(n)}$ as coordinate projection maps $Y^{(n)}(\omega^*) = \omega^*_n \in \mathbf{D}^{(n)}$ for $\omega^* = (\omega^*_2, \omega^*_3, \dots) \in \mathbf{D}^*$. Thus, 
\begin{align}
\mu^*\left(Y^{(n)} = d^{(n)}_j\right) = \mu^{(n)}\left( \{d^{(n)}_j\} \right) = \T^{(1)}_{n,j} \quad \text{for any \(n \geq 2\)}. \notag
\end{align}
\end{rem}
We can also study the Shannon entropy of each \emph{near}-log-concave probability distribution from $\T^{(1)}$. First, we denote the Shannon sum as follows:
\begin{align}
\mathcal{S}^{(y,n)} = -\sum_{j=1}^{n+1} \log\left(\T^{(y)}_{n,j}\right)\T^{(y)}_{n,j} \label{shannonsum}
\end{align}
for every $n\geq 2$ and any finite $y\in\R_+\setminus\{0\}$. We can take (\ref{shannonsum}) as a function $h : \R^{n+1}_+ \rightarrow \R$ where $\mathcal{S}^{(y,n)} = h(\T^{(y)}_{n,*})$ for every $n\geq 2$ and any finite $y\in\R_+\setminus\{0\}$. For what follows, we specifically denote
\begin{align}
\mathcal{S}^{(n)} \triangleq \mathcal{S}^{(2,n)}, \notag 
\end{align}
since $\T = \T^{(2)}$. We are now in position to state the following result.
\begin{lem}
\label{generalshannonsumresult}
Let $\mathcal{S}^{(y,n)}$ be as in (\ref{shannonsum}). Then,
\begin{align}
\mathcal{S}^{(y,n)} =\left(\frac{y}{2}\right)^n\mathcal{S}^{(n)}  + y^n\left(\log(2) - \log(y)\right)n, \notag
\end{align}
for every $n\geq 2$ and any finite $y\in\R_+\setminus\{0\}$.
\end{lem}
\begin{proof}
Using (\ref{scaledarray}) and (\ref{shannonsum}), we have the following:
\begin{align}
\mathcal{S}^{(y,n)} &= - \sum_{j=1}^{n+1} \left( \frac{y}{2}\right)^n \T_{n,j}\left(\log\left(\T_{n,j}\right) + n\log\left(\frac{y}{2}\right) \right) \notag \\ 
&= -\left(\frac{y}{2}\right)^n \sum_{j=1}^{n+1} \T_{n,j}\log\left(\T_{n,j}\right) - \left( \frac{y}{2}\right)^nn\log\left(\frac{y}{2}\right) \sum_{j=1}^{n+1} \T_{n,j} \notag \\
&= \left(\frac{y}{2}\right)^n \mathcal{S}^{(2,n)} - \left( \frac{y}{2}\right)^nn\log\left(\frac{y}{2}\right)2^n \label{shannonsimplificationline} \\
&= \left(\frac{y}{2}\right)^n \mathcal{S}^{(n)} - y^n n\log\left(\frac{y}{2}\right), \notag
\end{align}
for every $n\geq 1$ and any finite $y\in\R_+\setminus\{0\}$, where (\ref{shannonsimplificationline}) follows from Proposition \ref{summationpowertwo}.
\end{proof}
\begin{lem}
\label{shannonentropyspecificres}
The Shannon entropy is given by
\begin{align}
\mathcal{S}^{(1,n)} = 2^{-n}\mathcal{S}^{(n)}  + n\log(2), \label{shannonexpressionformula}
\end{align}
for every $n\geq 2$. Hence, $-2^nn\log(2) \leq \mathcal{S}^{(n)} < 0$ for every $n\geq 2$.
\end{lem}
\begin{proof}
Since $\mathcal{S}^{(1,n)} = -\sum_{j=1}^{n+1} \log(\T^{(1)}_{n,j})\T^{(1)}_{n,j}$, each $\mathcal{S}^{(1,n)}$ is Shannon entropy for $n\geq 2$, and (\ref{shannonexpressionformula}) follows from Lemma \ref{generalshannonsumresult}. For the boundary condition statement, since Shannon entropy $\mathcal{S}^{(1,n)}$ satisfies $ 0 \leq \mathcal{S}^{(1,n)} \leq \log(n+1)$, we have
\begin{align}
-2^nn\log(2) \leq \mathcal{S}^{(n)} \leq 2^n \left(\log(n+1) - n\log(2)\right) \label{initialboundary}
\end{align}
for every $n\geq 2$. The upper boundary in (\ref{initialboundary}) can be further tightened by observing that $\inf(\T) = 1$ and that $x \mapsto - x\log(x) \leq 0$ for any $x\geq 1$. Hence, 
\begin{align}
-2^nn\log(2) \leq \mathcal{S}^{(n)} \leq 0, \notag 
\end{align}
for every $n\geq 1$. Thus, we must have
\begin{align}
-2^nn\log(2) \leq \mathcal{S}^{(n)} < 0, \notag
\end{align}
for every $n \geq 2$, which gives the boundary statement.
\end{proof}
We are now in position to state the following result.
\begin{prop}
\label{shannonpropallsums}
Let $\mathcal{S}^{(y,n)}$ be as in (\ref{shannonsum}). Then,
\begin{align}
\mathcal{S}^{(y,n)} =y^n\left( \mathcal{S}^{(1,n)} - n\log(y) \right), \notag
\end{align}
for every $n\geq 2$ and any finite $y\in\R_+\setminus\{0\}$.
\end{prop}
\begin{proof}
The statement follows from Lemma \ref{generalshannonsumresult} and Lemma \ref{shannonentropyspecificres}.
\end{proof}
Using Proposition \ref{shannonpropallsums}, one can express every Shannon sum over $\T^{(y)}$ in terms of the Shannon entropy over $\T^{(1)}$ for any finite $y\in\R_+\setminus\{0\}$, which we shall later study numerically. Finally, we present a way to map $\T^{(1)}$ onto infinitely many right-stochastic matrices, each satisfying the three state-transition characteristics mentioned in the beginning of this paper -- the reader may prefer to review these characteristics for what follows.
We define
\begin{align}
n^* = \min\{n \geq 1 \, : \, \T^{(1)}_{n,n} - \T^{(1)}_{n,1}(n+2) \geq 0\}. \label{smallestnsatisfy}
\end{align}
For the next statement, let $\mathcal{M}^{\infty}$ be the space of all infinite-dimensional right-stochastic matrices that satisfy the aforementioned three state-transition  characteristics. In addition, let 
\begin{align}
\mathcal{Q} = \prod_{k\in\N_+} [0,1] \notag 
\end{align}
be the Tychonoff cube. We are now in position to state the following result.
\begin{prop}
\label{propstochastictransformation}
Let $n^*$ be defined as in (\ref{smallestnsatisfy}) for every $n\geq 1$. For any $\mathbf{\lambda} = [\lambda_1,\lambda_2,...] \in \mathcal{Q}$, define $\T^{(1)}(\mathbf{\lambda})$ as follows:
\begin{equation}
\label{infinitemanystochasticconstruct}
\T^{(1)}_{n,j}(\lambda_n) =
\begin{cases}
\T^{(1)}_{n,j} & n < n^* \, ,\, 1 \leq j \leq n+1 \\ 
\notag\\
\T^{(1)}_{n,j} - \lambda_n\T^{(1)}_{n,1} & n \geq n^* \, ,\, 1 \leq j \leq n  \\
\notag \\
\T^{(1)}_{n,j} + n\lambda_n\T^{(1)}_{n,1} & n \geq n^* \, ,\, j = n+1  \\
\end{cases}
\end{equation}
for every $n\geq 0$ and $1\leq j \leq n+1$. Then $\T^{(1)}(\mathbf{\lambda}) \in \mathcal{M}^{\infty}$ for any $\lambda\in\mathcal{Q}$.
\end{prop}
\begin{proof}
We have $\sup(\T^{(1)}_{n,*})=\T^{(1)}_{n,n}$ for every $n\geq 1$. Also, $[\T^{(1)}_{n,n},\ldots,\T^{(1)}_{n,1}]$ is monotonically non-increasing for every $n\geq 1$ and $\T^{(1)}_{n,n+k} = 0$ for any $k\geq 2$ and every $n\geq 1$. Hence, $\T^{(1)}\in\mathcal{M}^{\infty}$. Next, (\ref{smallestnsatisfy}) can be written as follows:
\begin{align}
n^* =  \min\{n \geq 1 \, : \, (\T^{(1)}_{n,n} - \T^{(1)}_{n,1}) - (\T^{(1)}_{n,n+1} + n\T^{(1)}_{n,1}) \geq 0 \}, \label{smallestnsatisfyexpanded}
\end{align}
since we have $\T^{(1)}_{n,1} = \T^{(1)}_{n,n+1}$ for every $n\geq 1$. Therefore, having $\sup(\T^{(1)}_{n,*})=\T^{(1)}_{n,n}$, we must have the following relation:
\begin{align}
\T^{(1)}_{n,n} - \lambda_n\T^{(1)}_{n,1} \geq \T^{(1)}_{n,n+1} + n\lambda_n\T^{(1)}_{n,1} \hspace{0.1in} \text{for $n\geq n^*$} \label{supremumrelationstatetonext}
\end{align}
for any $\lambda_n\in[0,1]$. Moreover, since $\lambda_n\geq 0$, $[\T^{(1)}_{n,n} - \lambda_n\T^{(1)}_{n,1},\ldots,\T^{(1)}_{n,1} - \lambda_n\T^{(1)}_{n,1}]$ is monotonically non-increasing for every $n\geq n^*$. Therefore, using (\ref{supremumrelationstatetonext}), we also have 
\begin{align}
\sup(\T^{(1)}_{n,*}(\lambda_n))=\T^{(1)}_{n,n}(\lambda_n), \notag 
\end{align}
for any $\lambda_n\in[0,1]$ for every $n\geq n^*$. From (\ref{infinitemanystochasticconstruct}) we see $\T^{(1)}_{n,n+k}(\lambda_n) = 0$ for any $k\geq 2$ and every $n\geq 1$. Finally, we have the following:
\begin{align}
\sum_{j=1}^n (\T^{(1)}_{n,j} - \lambda_n\T^{(1)}_{n,1} ) &= 1 - (\T^{(1)}_{n,n+1} + n\lambda_n\T^{(1)}_{n,1}) \, \Rightarrow \, || \T^{(1)}_{n,*}(\lambda_n) || = 1, \notag
\end{align}
for any $\lambda_n\in[0,1]$ and every $n\geq n^*$. Thus, $\T^{(1)}(\mathbf{\lambda}) \in \mathcal{M}^{\infty}$ for any $\lambda\in\mathcal{Q}$.
\end{proof}
\begin{rem}
For $\mathbb{\lambda} = [0,0,...] = \mathbf{0}$ being all-zeros, we have $\T^{(1)}(\mathbf{0}) = \T^{(1)}$.
\end{rem}
From Proposition \ref{propstochastictransformation}, since $\T^{(1)}(\mathbf{\lambda}) \in \mathcal{M}^{\infty}$ for any $\lambda\in\mathcal{Q}$, we have uncountably many right-stochastic matrices satisfying the three state-transition characteristics, forming an infinite family of candidates to model random systems with such behaviour. This is certainly only one such transformation of $\T^{(1)}$ and many others can be proposed.
\begin{rem}
The solution to (\ref{smallestnsatisfy}) is $n^* = 5$.
\end{rem}

\section{Numerical Study}
We shall motivate that further properties not included or studied in this paper reside within the structure of $\T$ -- and accordingly $\T^{(y)}$ in that spirit for any finite $y\in\R_+\setminus\{0\}$ -- which can be worked out analytically and/or numerically. 
For what follows, we shall focus our attention on $\T^{(1)}$ due to its probabilistic nature, which forms the initial motivation of this paper. First, in correspondence to the three main state-transition characteristics mentioned in the Introduction, we collate three properties of $\T^{(1)}$ below, which are already shown in the previous section.

\begin{rem}
\label{threemaincharacteristics}
The following properties hold for the matrix form of $\T^{(1)} \in \mathcal{M}^{\infty}$:
\begin{enumerate}
\item $\sup(\T^{(1)}_{n,*})=\T^{(1)}_{n,n}$ for every $n\geq 1$
\item $[\T^{(1)}_{n,n},\ldots,\T^{(1)}_{n,1}]$ is monotonically non-increasing for every $n\geq 1$
\item $\T^{(1)}_{n,n+k} = 0$ for any $k\geq 2$ and every $n\geq 1$
\end{enumerate}
\end{rem}
Next, in order to aid the reader with the anchor integer progression towards making it visually available, we shall also provide the derived values of the first few rows and columns of the triangular array $\T$, given below. 
\\
\begin{figure}[htbp] 
  \centering
  \includegraphics[width=0.67\linewidth]{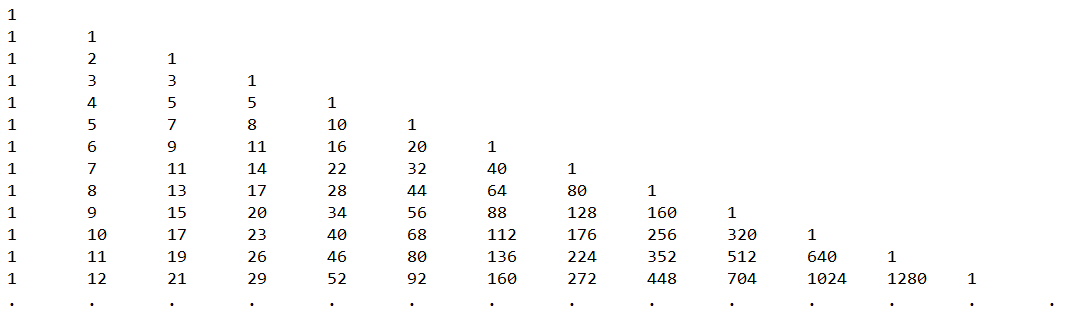}
  \caption{Terms of the triangular array $\T$}
  \label{figure1}
\end{figure}
\\
The following charts demonstrate (i) the probability mass functions as one progresses from one row to the next over the normalised array $\T^{(1)}$, (ii) the behaviour as one progresses from one finite column to the next over $\T^{(1)}$. 
\\
\begin{figure}[!ht]
  \centering
  \includegraphics[width=0.45\linewidth]{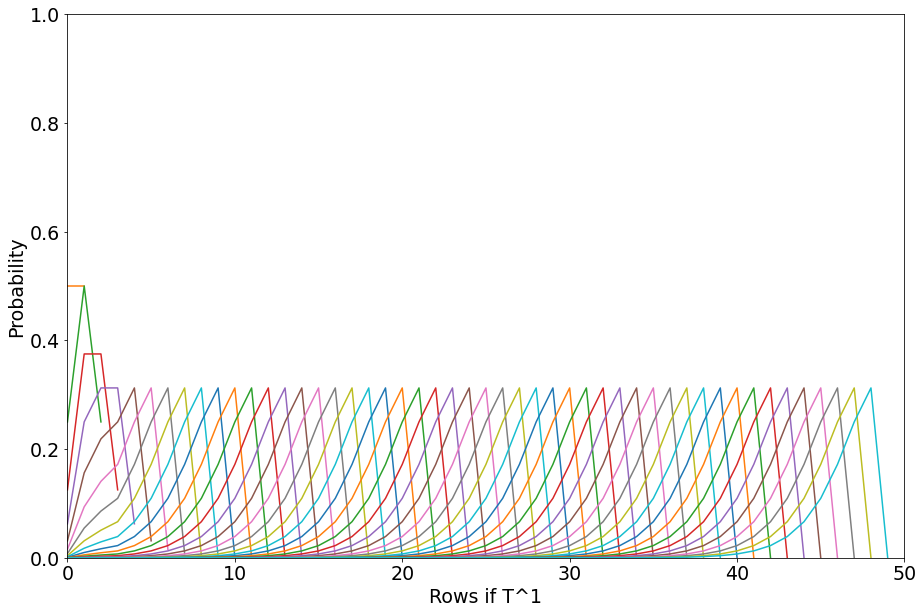}%
  \qquad
  \includegraphics[width=0.45\linewidth]{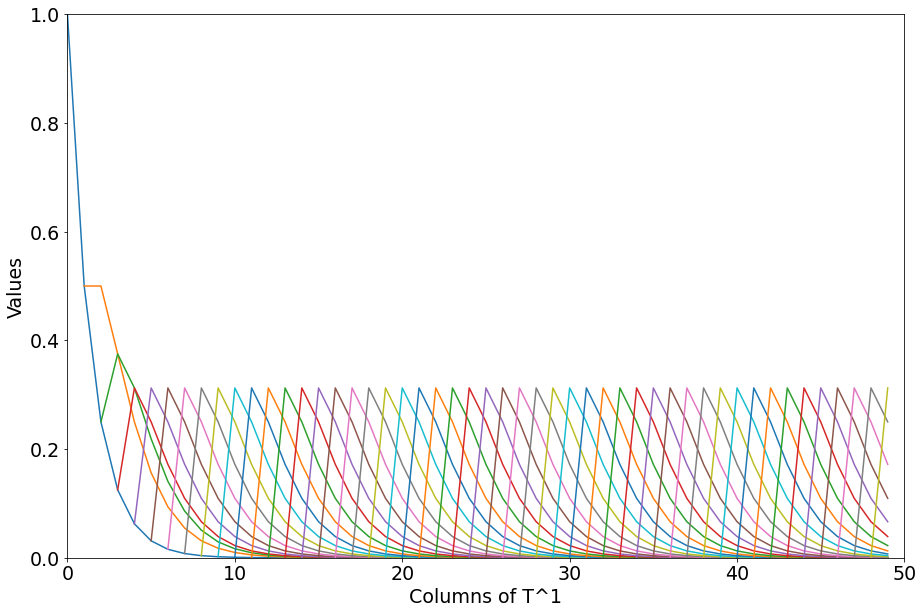}
  \caption{Left: First 50 probability mass functions from $\T^{(1)}$. Right: First 50 columns of $\T^{(1)}$}
  \label{figure2}
\end{figure}
\\
Here, we point to Remark \ref{threemaincharacteristics} for patterns that can be observed from the left-panel of the chart. In addition, one can observe a form of symmetry across the rows and columns of $\T^{(1)}$ from Figure 2; the rows (i.e. probability mass functions) are left-skewed and the columns are right-skewed. 

A central background interest in our study is long-term behavior of statistical moments. Although expanding arrays may exhibit unbounded dispersion as the underlying state-space grows (e.g. symmetric binomial distribution induced by Pascal's triangle), the construct presented in this paper demonstrates a structural decoupling between a divergent mean and a convergent variance. More specifically, if the state-space of an associated near-log-concave random variable $X^{(n)} : \Omega \rightarrow \mathbf{D}^{(n,*)}$ for any $n\geq 2$ is given as the set of integers:
\begin{align}
\mathbf{D}^{(n,*)} \triangleq \{1,\ldots,n+1\}, \notag
\end{align}
then the $m$-th moment of $X^{(n)}$ is given by the folllowing sum:
\begin{align}
\langle\textbf{M}\rangle_{\mathbf{D}^{(n,*)}}^{(m)} = \sum_{j=1}^{n+1} j^m\mu^{(n)}(\{j\}), \label{momentforspecific}
\end{align}
for any $n\geq 2$ and $m\geq 1$. Now, using (\ref{algorithmentire}), Proposition (\ref{summationpowertwo}) and Proposition \ref{nearlogconcaverandomvariableprop}, we have
\begin{align}
\P(X^{(n)} = j) = 2^{-n}\T_{n,j} = \frac{5 + 3(n-j)}{2^{n-j+4}} = \mu^{(n)}(\{j\}), \label{reducedprobexp}
\end{align}
for $4\leq j \leq n$. From (\ref{momentforspecific}), we therefore have the first and second moments:
\begin{align}
&\langle\textbf{M}\rangle_{\mathbf{D}^{(n,*)}}^{(1)} = n - \frac{7}{4} + \frac{3n + 1}{2^n}, \label{firstmomentcalc} \\
&\langle\textbf{M}\rangle_{\mathbf{D}^{(n,*)}}^{(2)} = n^2 - \frac{7n}{2} + \frac{27}{4} + \frac{n^2 + 6n -13}{2^n}. \label{secondmomentcalc}
\end{align}
Therefore, using (\ref{firstmomentcalc}) and (\ref{secondmomentcalc}), the variance satisfies
\begin{align}
\text{Var}^{\P}\left[X^{(n)}\right] = \E^{\P}\left[(X^{(n)})^2\right] - \E^{\P}\left[(X^{(n)})\right]^2 = \frac{59}{16} - \frac{10n^2 - 29n + 19}{2^{n+1}} - \frac{9n^2 + 6n + 1}{4^n}. \label{variancestatic}
\end{align}
This suggests that the variance of $X^{(n)}$ converges to a constant as $n\rightarrow\infty$. From (\ref{variancestatic}), we can state the following.
\begin{rem}
Let $X^{(n,q)} : \Omega \rightarrow \mathbf{D}^{(n,q)}$ be a near-log-concave random variable with $\mathbf{D}^{(n,q)} = \{q,\ldots,n+q\}$ for any $q\in\N_+$, so that 
\begin{align}
\langle\textbf{M}\rangle_{\mathbf{D}^{(n,q)}}^{(m)} = \sum_{j=1}^{n+1} (j + q -1)^m\mu^{(n)}(\{j\}), \notag
\end{align}
for any $n\geq 2$ and $m\in\{1,2\}$. Then, $\text{Var}^{\P}\left[X^{(n,q)}\right]  \rightarrow 59/16$ as $n \rightarrow \infty$, for any $q\in\N_+$.
\end{rem}
This work essentially shows how a simple combinatorial rule can push variance towards a steady-state limit even as the underlying support grows indefinitely. We shall expand the asymptotic analysis further by defining a transformed random variable:
\begin{align}
Z^{(n)} = n - X^{(n)}, \notag
\end{align}
for $n\geq 4$, which renders the support of $Z^{(n)}$ to be $\{-1,\ldots,n-1\}$. Let $z=n-j$, so that the condition $4\leq j \leq n$ maps to $0\leq z \leq n-4$. Thus, using (\ref{reducedprobexp}), we have
\begin{align}
\P(Z^{(n)} = z) = \P(X^{(n)} = n - z) = \frac{5 + 3z}{2^{z+4}} \,\,\,\,\, \text{for $0\leq z \leq n-4$}. \label{reducedprobexpshifted}
\end{align}
When $z=-1$, we have $j=n+1$, and therefore, using $\T^{(1)}_{n,n+1} = 2^{-n}$, we get
\begin{align}
\lim_{n\rightarrow\infty}\P(Z^{(n)} = -1) = \lim_{n\rightarrow\infty}\frac{1}{2^n} = 0. \notag
\end{align}
In addition, since the upper boundary $n-4$ in (\ref{reducedprobexpshifted}) dominates any fixed $z$ in the limit as $n\rightarrow\infty$, we have the following:
\begin{align}
\lim_{n\rightarrow\infty}\P(Z^{(n)} = z) = \frac{5 + 3z}{2^{z+4}} \,\,\,\,\, \text{for $z\geq 0$}. \label{keylimitmassfunction}
\end{align}
Note that (\ref{keylimitmassfunction}) is non-negative for every $z\geq 0$ and
\begin{align}
\sum_{z=0}^{\infty} \frac{5 + 3z}{2^{z+4}} = \frac{1}{16}\left(5 \sum_{z=0}^{\infty} \frac{1}{2^z} + 3\sum_{z=0}^{\infty} \frac{z}{2^z} \right) = 1, \label{keylimitmassfunctionsumidentity}
\end{align}
which follows using standard geometric series identities given by
\begin{align}
\sum_{z=0}^{\infty} \frac{1}{2^z} = 2 \,\,\,\,\, \text{and} \,\,\,\,\, \sum_{z=0}^{\infty} \frac{z}{2^z} = 2. \notag
\end{align}
Thus, there exists a random variable $Z$ with infinite support $\N_0 \cup \{-1\}$ having zero mass at $-1$, such that $Z^{(n)}$ converges in distribution
\begin{align}
Z^{(n)} \xrightarrow{\text{d}} Z \,\,\,\, \text{as $n\rightarrow\infty$}, \notag
\end{align}
with its probability mass function given by
\begin{align}
\P(Z = z) = \frac{5 + 3z}{2^{z+4}} \,\,\,\,\, \text{for $z\geq 0$}, \label{limitrandomvariable}
\end{align}
from (\ref{keylimitmassfunction}). Hence, in relation to (\ref{firstmomentcalc}), (\ref{secondmomentcalc}) and (\ref{variancestatic}), we get
\begin{align}
\E^{\P}[Z] = \frac{7}{4}, \,\,\,\,\, \E^{\P}[Z^2] = \frac{27}{4}, \,\,\,\,\, \text{Var}^{\P}\left[Z\right] = \frac{59}{16}, \notag
\end{align}
which shows that the probability mass function of $Z^{(n)}$ converges to a stable, non-dispersive limiting distribution. One way to interpret this observation is how the sequence of probability mass functions from $\T^{(1)}$ progressively mirrors a \emph{wave} that organizes itself as a self-reinforcing packet inside a discrete channel as the mean advances (i.e. slides) with increasing $n$. Thus, as the variance converges while the relative peak profile remains structurally pinned, the distribution behaves as a non-dispersive wave pulse propagating forward through its corridor without flattening out, but instead, dynamically converging to a stable shape in the form of a limiting probability mass function. Accordingly, $\T^{(1)}$ establishes a balance reminiscent of soliton dynamics in physics.

Next, we present the evolution of the Shannon entropy $\mathcal{S}^{(1,n)}$ over the rows of $\T^{(1)}$, and the Shannon sums $\mathcal{S}^{(y,n)}$ for some samples of $y\in[0.5,1]$, leading up to $\mathcal{S}^{(1,n)}$. The numerical trajectories in Figure 3 demonstrate that the Shannon entropy evolution over $\T^{(1)}$ is (i) monotonically non-decreasing and (ii) convergent.
\\
\begin{figure}[!ht]
  \centering
  \includegraphics[width=0.45\linewidth]{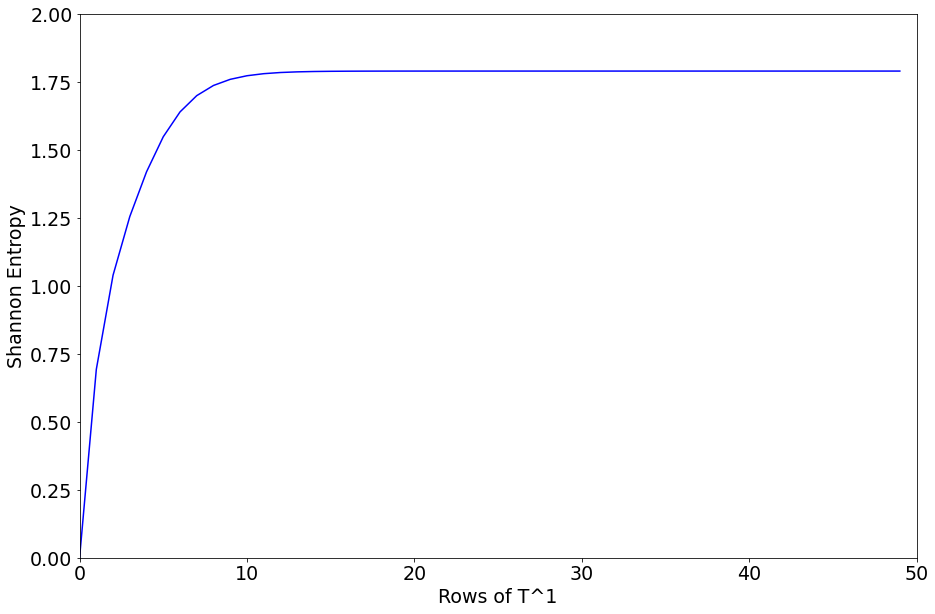}%
  \qquad
  \includegraphics[width=0.45\linewidth]{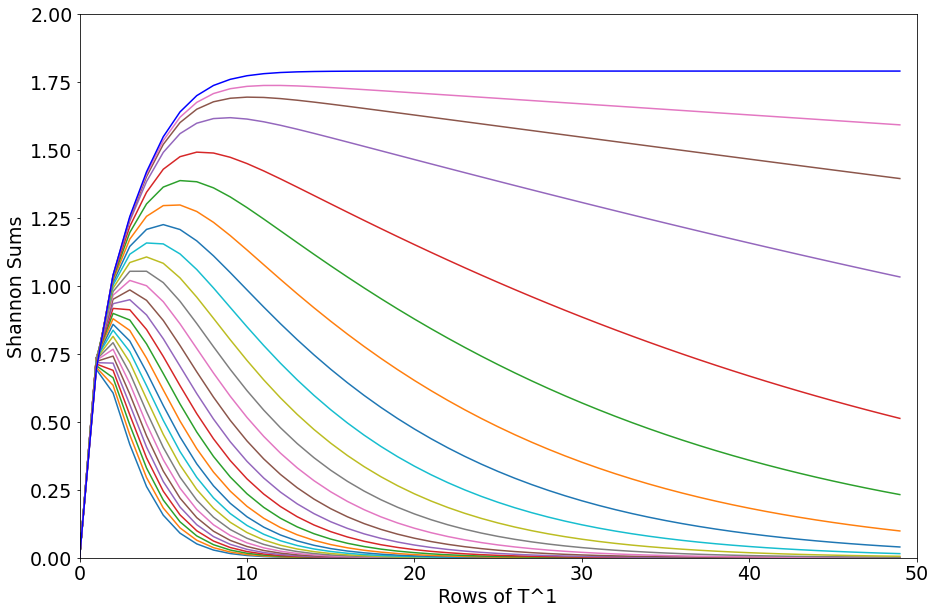}
  \caption{Left: Shannon entropy $\mathcal{S}^{(1,n)}$. Right: Shannon sums $\mathcal{S}^{(y,n)}$ for samples $y\in[0.5,1]$}
  \label{figure3}
\end{figure}
\\
Regarding convergence of Shannon entropy as $n\rightarrow\infty$, we can use the limit random variable $Z$ and its probability mass function (\ref{limitrandomvariable}), which provide us with the following:
\begin{align}
\lim_{n\rightarrow\infty} \mathcal{S}^{(1,n)} = -\sum_{z=0}^{\infty} \log\left( \frac{5 + 3z}{2^{z+4}} \right)\left(\frac{5 + 3z}{2^{z+4}}\right) \approx 1.7907694690801 \ldots\notag
\end{align}
Moreover, numerical evaluations indicate that for any positive $y\leq 1$, Shannon entropy $\mathcal{S}^{(1,n)}$ is the only monotonically non-decreasing progression when compared to $\mathcal{S}^{(y,n)}$ with $y < 1$.

\begin{rem}
The limiting profile obtained above suggests a broader family of probability mass functions of the form
\begin{align}
f(z \, ; a,b,c) =\frac{a+bz}{2^{z+c}} \,\,\,\, \text{$z\geq 0$}, \notag
\end{align}
for $a,b,c\in\N_0$ and $c\geq 1$, subject to the normalization condition
\begin{align}
a+b=2^{c-1}. \notag
\end{align}
Note that $f(z \, ; a,b,c) \geq 0$ for every $z\geq 0$ and
\begin{align}
\sum_{z=0}^{\infty} f(z \, ; a,b,c) = 1. \notag
\end{align}
The present construction corresponds to the particular choice $(a,b,c)=(5,3,4)$. It is natural to ask which members of this family arise asymptotically from analogous triangular-array progressions that give rise to discrete non-dispersive propagation. We leave this broader question for future work.
\end{rem}

\section{Conclusion}
We introduced and studied a sequence progression algorithm that produces an infinite triangular array with numerous properties as presented in the main body of this work. Among these properties, the triangular array defines infinitely many degree sequences of multigraphs without loops, produces unimodal near-log-concave sequences, hosts a totally non-negative Toeplitz matrix of order two, and naturally lends itself to a probabilistic setup through what we call near-log-concave random variables and their induced probability measures. We proved several asymptotic properties of the proposed triangular array to understand its convergent features. We showed that the triangular array can be taken as an anchor integer structure to construct infinitely many lower-triangular stochastic matrices satisfying the three aforementioned state-transition characteristics. A natural direction for future work is to determine the conditions under which these state-transition characteristics can support discrete non-dispersive wave propagation patterns across ordered state-spaces. A mathematical synthesis of advancing divergent mean alongside bounded convergent variance (and, more strongly, the existence of a stable probability mass function in the limit) suggests that such combinatorial rules may provide an analytical approach to such investigations. In this spirit, we shall examine systems with such state-transition characteristics from a more holistic viewpoint, whereby the sequence progression presented in this paper forms a promising example. We believe that studying infinite sequence progressions hosting particular characteristics as triangular arrays has the potential to provide a viable path to generate constructs for exploring certain patterns observed in nature.

\end{document}